\documentclass[12pt]{article}
\usepackage{bbm}
 \usepackage{amssymb}
\usepackage{amssymb, amsthm, amsmath, amscd}
\usepackage[usenames,dvipsnames]{color}

\newtheorem{thm}{Theorem}[section]
\newtheorem{lem}[thm]{Lemma}
\newtheorem{prop}[thm]{Proposition}

\theoremstyle{definition}

\theoremstyle{definition}
\newtheorem{df}[thm]{Definition}
\theoremstyle{definition}
\newtheorem{rem}[thm]{Remark}
\newtheorem{nota}[thm]{Notation}
\theoremstyle{definition}

\newcommand{\red}{\textcolor{red}}

\renewcommand{\phi}{\varphi}

\definecolor{purple}{RGB}{150,10,200} 

\newcommand{\CAs}{$C^*$-algebras}

\newcommand{\Ped}{{\rm Ped}}
\newcommand{\Cu}{{\rm Cu}}
\newcommand{\LAff}{{\rm LAff}}

\newcommand{\QT}{{\rm QT}}

\newcommand{\TQT}{\widetilde{\rm QT}}
\newcommand{\Lsc}{{\rm Lsc}}

\newcommand{\wtilde}{\widetilde}
 
\newcommand{\wh}{\widehat} 
\newcommand{\ol}{\overline}  
\newcommand{\nzero}{\setminus\{0\}}  
\newtheorem*{theorema}{Theorem A}

\newcommand{\N}{\mathbb{N}}
\newcommand{\K}{\mathbb{K}}
\newcommand{\Z}{\mathbb{Z}}

\newcommand{\R}{\mathbb{R}}
\newcommand{\C}{\mathbb{C}}

\numberwithin{equation}{section}

\newcommand{\Aff}{\operatorname{Aff}}

\newcommand{\Her}{\mathrm{Her}}

\newcommand{\dt}{\delta}
\newcommand{\ep}{\varepsilon}

\newcommand{\rforal}{\,\,\,{\rm for\,\,\,all}\,\,\,}
\newcommand{\CA}{$C^*$-algebra}

\newcommand{\wtd}{\widetilde}

\newcommand{\beq}{\begin{eqnarray}}
\newcommand{\eneq}{\end{eqnarray}}

\usepackage{amsfonts}
\usepackage{mathrsfs}
\usepackage{textcomp}
\usepackage[all]{xy}
\usepackage{enumerate}
\usepackage{hyperref} 

\title{From Stable Rank One to Real Rank Zero: 
\\ {\large A Note on Tracial Approximate Oscillation Zero}}  
\author{Xuanlong Fu\footnote{Department of Mathematics,  Tongji University. Email: xuanlongfu@tongji.edu.cn.}}
\date{ }
\begin{document}

\maketitle

\begin{abstract} 
We present a 
relation between stable rank one and real rank zero via the method of tracial oscillation. 
Let $A$ be a  simple separable $C^*$-algebra of stable rank one. 
We show that $A$ has tracial approximate oscillation zero and,
as a consequence, the tracial sequence algebra $l^\infty(A)/J_A$ has real rank zero, where $J_A$ is the trace-kernel ideal 
with respect to 2-quasitraces. 

We also show that for a \CA\ $B$ that has non-trivial 2-quasitraces, $B$ has tracial approximate oscillation zero is equivalent to $l^\infty(B)/J_B$ has real rank zero. 


\end{abstract}

\section{Introduction}

Stable rank one (\cite{{Rff}}) and real rank zero (\cite{BP91}) are regularity properties of \CAs\ that play important role in the classification of \CAs. 
\CAs\ that can be classified by the Elliott invariant are ${\cal Z}$-stable (i.e., $A\cong A\otimes {\cal Z}$), where 
 ${\cal Z}$ is the Jiang-Su algebra constructed by Jiang and Su in \cite{JS99}. ${\cal Z}$ is the unique  simple separable unital nuclear infinite-dimensional \CA\ that satisfies UCT and has the same Elliott invariant as complex numbers. 
Simple  ${\cal Z}$-stable \CAs\ 
are either of stable rank one or of real rank zero. 
Gong-Jiang-Su showed that  simple unital ${\cal Z}$-stable \CAs\ are  either stably finite or  purely infinite (\cite{GJS}). 
Zhang showed that simple purely infinite \CAs\ are have real rank zero (\cite{Z90}). 
R{\o}rdam  showed that simple unital ${\cal Z}$-stable \CAs\ are either of stable rank one or  purely infinite (\cite{Ror04JS}). 
In 2021, the author together with Li and Lin proved that 
simple  ${\cal Z}$-stable  (not necessary unital)   \CAs\  are either have stable rank one or purely infinite (\cite[Corollary 6.8]{FLL21}). 
See also \cite{Rob,FL2022} for related results on almost stable rank one. 
Non-${\cal Z}$-stable simple  \CAs\ also can have stable rank one. Villadsen's algebras of first type are such examples (\cite{V98}). 
Elliott-Ho-Toms showed that simple diagonal AH-algebras have stable rank one (\cite{EHT}).

The surjectivity of the canonical map  $\Gamma: \Cu(A)\to {\rm LAff}_+\left(\TQT(A)\right)$ is also an important regularity property of \CAs, which is related to the computation of the Cuntz semigroups of \CAs. 
{{N.~P.~Brown first raised the question when is $\Gamma$ being  surjective (see the remark after Question 1.1 of \cite{T20}).}} 
Thiel showed that $\Gamma$ is surjective for simple separable unital  \CAs\ of stable rank one (\cite[Theorem 8.11]{T20}). 
Antoine-Perera-Robert-Thiel showed that $\Gamma$ is surjective for simple separable stable rank one \CAs\ (\cite[Theorem 7.13]{APRT}).
See \cite{BPT, ERS, DT} for more works on the surjectivity of $\Gamma.$ 

In order to study the relations between stable rank one and other regularity properties, the author and Lin introduced the notion of tracial approximate oscillation zero in \cite{FLosc}. 
{{As pointed out in the first paragraph of the introduction of \cite{FLosc}, tracial approximate oscillation zero has a background in the small boundary property in dynamical systems.}} 
Assume that $A$ is a simple separable \CA\ with nontrivial 2-quasitraces and $A$ has strict comparison.
We showed (in \cite[Theorem 1.1]{FLosc}) that $A$ has stable rank one 
is equivalent to  $A$ has tracial approximate oscillation zero, and 
is also equivalent to that 
the canonical map $\Gamma$ is surjective and $A$ almost has stable rank one. 

Latter, by using tracial approximate oscillation zero, Lin proved that for a simple \CA\ $B$ with  $\Cu(B)\cong \Cu(B\otimes {\cal Z}),$  $B$ is either of purely infinite or has stable rank one (\cite{LinJFA}). 
Let $A$ be a simple separable non-elementary amenable stable rank one  \CA\ 
with the  $\tilde T(A)\setminus\{0\}\neq \emptyset.$ Tracial approximate oscillation zero also plays an important role in 
Lin's work on proving the Toms-Winter conjecture for $A$ with the extreme boundary of $\tilde T(A)\nzero$ satisfying condition (C) (see \cite[Theorem 5.6]{LinAdv}) and for $A$ with  $\tilde T(A)\nzero$ has  $\sigma$-compact countable-dimensional extremal boundary ({{see \cite[Theorem 1.1]{L22}}}). {{See \cite{
CETWW21, KR14, 
MS12, MS14, Ror04JS, T20, W12pure} for more works on the Toms-Winter conjecture.}}

There are plenty of examples that have stable rank one but not have real rank zero, e.g., the Jiang-Su algebra and Villadsen's algebras of first type. 
Still, our main theorem in this paper is to show that separable simple stable rank one \CAs\ 
have tracial approximate oscillation zero, and 
they are tracially  real rank zero in the following  sense
(where $J_A$  is the trace-kernel ideal with respect to 2-quasitraces): 

\begin{theorema}\label{shier-29-1}
{\rm (Theorem \ref{jiu18-a1})}
Let $A$ be a separable
simple \CA\ of stable rank one, 
then $A$ has tracial approximate oscillation zero. 
Moreover,  $l^\infty(A)/
{{J_A}}$ has real rank zero. 
\end{theorema}


{{The proof of above theorem is based on an observation that  Antoine-Perera-Robert-Thiel's results (\cite{APRT}) and Lin's techniques  (\cite{LinJFA}) can be fitted together, via a notion called hereditary surjective (see Definition \ref{HER-surj}).}}
This result generalizes \cite[Theorem 5.10]{FLosc} to \CAs\ without strict comparison. 
In particular, for a simple separable \CA\ $A,$ if $A$ is a  diagonal AH algebra, or  $A=C(X)\rtimes \Z^d,$
where $(X,\Z^d)$ is  a minimal free topological dynamical system,
then $A$ has tracial approximate oscillation zero and $l^\infty(A)/J_A$ has real rank zero (see Theorem \ref{shier26-1}). 
We also show that for separable 
simple \CA\ $A$ with $\QT(A)\neq \emptyset,$ $A$ has tracial approximate oscillation zero if and only if  $l^\infty(A)/J_A$ has real rank zero (Theorem \ref{1204-6}). 

This paper is organized as follows: Section 2 is a preliminary that mainly devoted to explain the relation between different versions of quasitraces. In Section 3 we introduce tracial  oscillation and show that tracial approximate oscillation zero is equivalent to $l^\infty(A)/J_A$ has real rank zero. 
In Section 4 we introduce the main technical device called hereditary surjective canonical map and show that stable rank one implies hereditary surjectivity. In Section 5 we prove that hereditary surjectivity implies tracial approximate oscillation zero, which gives a proof to the Theorem A. In Section 6 we give some applications to diagonal AH-algebras and $\Z^d$-crossed products. 

{\bf Acknowledgement:} 
The author is grateful to Huaxin Lin for many helpful conversations.

\section{Preliminaries}


\begin{nota} 
The set of all positive integers is denoted by $\N.$ 
The set of all non-negative real numbers is denoted by $\R_+.$
The set of all compact operators on a separable 
infinite-dimensional Hilbert 
space is denoted by $\K.$ 
Let $\{e_{i,j}\}$ denote a set of matrix units of $\K.$
%
Let  $(X,d)$ 
be a metric space,  let $x,y\in X$ and let $\ep>0$.
We write $x\approx_{\ep}y$ if
$d(x,y)\le \ep$.


\end{nota}

\begin{nota}
Let $A$ be a $C^*$-algebra. 
Denote by $A^{1}$ the closed unit ball of $A,$ and 
by $A_+$ the set of all positive elements in $A.$
Put $A_+^{1}:=A_+\cap A^{1}.$ 
The set of all self-adjoint elements of $A$ is denoted by $A_{sa}.$ 
Let $a\in A_+.$ Let $\Her_A(a)$ (or just $\Her(a),$ when $A$ is clear)
be the hereditary $C^*$-subalgebra of $A$ generated by $a.$
The Pedersen ideal of $A$ is 
denoted by $\Ped(A),$ which is the minimal dense ideal of $A$ (\cite[5.6]{Pedbk}).  
Let ${\rm Ped}(A)_+= {\rm Ped}(A)\cap A_+,$
${\rm Ped}(A)^{1}= {\rm Ped}(A)\cap A^{1}$ and ${\rm Ped}(A)_+^{1}={\rm Ped}(A)\cap A_+^{1}.$ 

Let $\ep >0.$ Define a continuous function
$f_{\ep} 
: \R
\rightarrow [0,1]$ as following: 
$f_\ep (t)=0$ for $t\in(-\infty,\ep],$ 
$f_\ep (t)=1$ for $t\in[2\ep,+\infty),$
and $f_\ep$ is linear on $[\ep,2\ep].$

\end{nota}

\begin{prop}\label{jiu27-2}
{\rm (\cite[Lemma 2.2.]{KRadv}, \cite{RorUHF2})}
Let $A$ be a \CA, let $a,b\in A_+,$ and 
let $\ep>\|a-b\|.$
Then there is a \emph{contraction} $c\in A^1$ such that $(a-\ep)_+=c^*bc.$ 
\end{prop}

\begin{df}\label{Dcuntz}
Let $A$ be a \CA\
and let  $a, b\in A_+.$ 
We write $a \lesssim b$ if there are
$x_k\in A$
such that
$\lim_{k\rightarrow\infty}\|a-x_k^*bx_k\|=0$.
We write $a \sim b$ if $a \lesssim b$ and $b \lesssim a$  both hold. 
The Cuntz relation $\sim$ is an equivalence relation.
Set $\Cu(A)=(A\otimes \K)_+/\sim.$  
For $a\in (A\otimes \K)_+,$ let $[a]$ denote the Cuntz equivalence class corresponding to $a.$ The partial order on $\Cu(A)$ is given by the following: We write $[a]\le [b]$ whenever $a\lesssim b$ holds. 

Let $\iota: \K\otimes M_2(\C)\to \K$ be a $*$-isomorphism, which induces a $*$-isomorphism $\bar \iota:={\rm id}_A\otimes \iota: (A\otimes\K)\otimes M_2(\C)\to A\otimes\K.$ 
{{For $a,b\in (A\otimes \K)_+,$ define $[a]\oplus[b]:=[\bar \iota(a\otimes e_{1,1}+b\otimes e_{2,2})]\in\Cu(A).$}} With this (well-defined) addition, $\Cu(A)$ becomes a semigroup, which is called the Cuntz semigroup of $A.$ 
\end{df}

\begin{df}
Let $A$ be a \CA. A map $\phi: \Cu(A)\to [0,\infty]$ 
is called a functional, if $\phi$ is additive, 
order-preserving, $\phi(0)=0,$ and also
preserve the suprema of increasing sequences. 
Let ${\rm F}(\Cu(A))$ be  set of all functionals on $\Cu(A).$ 

The partial {{order}} on ${\rm F}(\Cu(A))$ is the canonical one: 
$f_1\le f_2$ if and only if $f_1(x)\le f_2(x)$ for all $x\in \Cu(A).$
${\rm F}(\Cu(A))$ is endowed with the topology that 
a net $\{\lambda_i\}\subset {\rm F}(\Cu(A))$ is converge to $\lambda\in  {\rm F}(\Cu(A)),$ if and only if 
$\limsup_i \lambda_i([(a-\ep)_+])\le \lambda([a])\le \liminf_i\lambda_i([a])$
for all $a\in (A\otimes\K)_+$ and all $\ep>0.$
\end{df}

\begin{df}\label{shi01-10}
(\cite[II.1.1]{BH})
Let $A$ be a pre-\CA. A quasitrace on $A$ is a map $\tau: A\to\C$ 
such that 
{\bf (1)} $0\le \tau(x^*x)=\tau(xx^*)$ for all $x\in A$; 
{\bf (2)} $\tau$ is linear on commutative $^*$-subalgebras of $A$; 
{\bf (3)} $\tau(a+ib)=\tau(a)+i\tau(b)$ for all $a,b\in A_{sa}.$

If $\tau$ can be extended to a quasitrace on $M_2(A),$ then $\tau$ is called a 2-quasitrace. 
\end{df}

By   \cite[II.2.3, II.2.5, II.4.1]{BH}, we have the following: 
All quasitraces on a \CA\ $A$  are continuous. When $\tau$ is a 2-quasitrace on $A,$ $\tau$ is order preserving, i.e., $a\le b$ implies $\tau(a)\le \tau(b),$ and $\tau$ can be extended to $M_n(A)$ for all $n\in\N.$ 

The following result is a folklore. See \cite[Proposition 2.7]{L22} for the case of 2-quasitraces. 
\begin{prop}\label{jiu25-3}
{\rm (cf. \cite[Proposition 2.7]{L22})}
Let $A$ be a \CA. 
Let $\tau: \Ped(A)\to\C$ be a quasitrace, 
then $\tau$ is lower semicontinuous on $\Ped(A)_+$. 
\end{prop}
\begin{proof}
Let $a\in \Ped(A)_+,$ 
then $C^*(a)\subset \Ped(A)$ 
(see \cite[5.6.2]{Pedbk}). 
By \cite[II.2.3]{BH}, $\tau$ is continuous on $C^*(a).$
Let $\ep>0,$ then there is $\dt>0$ such that 
$\tau(a)-\ep<\tau((a-\dt)_+).$ 
Let $b\in \Ped(A)_+$ with $\|a-b\|<\dt.$ 
By Proposition \ref{jiu27-2}, 
there is a contraction $r\in C^*(a,b)^1\subset\Ped(A)_+$
such that $(a-\dt)_+=r^*br.$
Then 
$
\tau(a)-\ep\le \tau((a-\dt)_+)
=\tau(r^*br)=\tau(b^{1/2}rr^*b^{1/2})\le \tau(b), 
$
which shows $\tau$ is lower semicontinuous.  
\end{proof}

Adopt the convention in \cite[2.7]{FLosc}, we have  the following: 
\begin{df}\label{Dqtr}
Let $A$ be a \CA. 
A densely  defined 2-quasitrace on $A\otimes \K$ is a 2-quasitrace 
$\tau:\Ped(A\otimes \K)\to\C.$ 
Denote by $\TQT(A)$ the set of all densely defined 2-quasitraces 
on $A\otimes \K.$  


The partial order on $\TQT(A)$ is the canonical one: 
For $\tau_1,\tau_2\in\TQT(A),$ we write 
$\tau_1\le \tau_2$ if $\tau_1(a)\le \tau_2(a)$ for all $a\in \Ped(A\otimes\K)_+.$
The topology on $\TQT(A)$ is defined by pointwise convergence: 
A net $\{\tau_i\}\subset \TQT(A)$ is converge to $\tau\in\TQT(A)$
if and only if 
$\lim_i\tau_i(a)=\tau(a)$ for all $a\in\Ped(A\otimes\K).$ 
\end{df}

\begin{rem}\label{shi05-r2}
{\bf (1)} 
A quasitrace 
$\tau:\Ped(A\otimes \K)\to\C$ is automatically a 2-quasitrace, due to the fact that $\K\otimes M_2\cong \K.$
{\bf (2)} If $\tau_1,\tau_2\in\TQT(A)$ and $\tau_1\le \tau_2,$ then $\tau_2-\tau_1\in \TQT(A).$ 
\end{rem}

\begin{df}\label{jiu21-d}
(\cite[4.1, p\ 984]{ERS})
Let $A$ be a \CA. 
A lower semicontinuous 
2-quasitrace on $A$ is 
a lower semicontinuous 
map $\tau:(A\otimes\K)_+\to [0,\infty]$ 
such that $\tau(0)=0,$  
$\tau(x^*x)=\tau(xx^*)$ for all $x\in A\otimes\K,$
$\tau(a+b)=\tau(a)+\tau(b)$ 
for all $a,b\in(A\otimes\K)_+$ with $ab=ba.$ 
Let $\QT_2(A)$ be the set of all lower semicontinuous 
2-quasitraces on $A$. 

The partial order on $\QT_2(A)$ is defined by the following: 
For $\tau_1,\tau_2\in\QT_2(A),$ we write 
$\tau_1\le \tau_2$ if $\tau_1(a)\le \tau_2(a)$ for all $a\in (A\otimes\K)_+.$
The topology on $\QT_2(A)$ is defined by the following: 
A net $\{\tau_i\}\subset \QT_2(A)$ is converge to $\tau\in  \QT_2(A),$ if and only if 
$
\limsup_i \tau_i((a-\ep)_+)\le \tau(a)\le \liminf_i\tau_i(a)
$
for all $a\in (A\otimes\K)_+$ and all $\ep>0.$

Let $\tau\in\QT_2(A).$ $\tau$ is called densely finite, 
if $\tau(a)<+\infty$ for all $a\in\Ped(A\otimes\K)_+.$ 
Let $\tau_\infty\in \QT_2(A)$ be the map such that $\tau_{\infty}(a)=\infty$ for all $a\in\K_+\nzero.$
\end{df} 

\begin{prop}\label{jiu27-3}
Let $A$ be a simple \CA. 
For any $\tau\in\QT_2(A)\setminus\{\tau_\infty\},$
$\tau$ is densely finite. 
\end{prop}
\begin{proof}
Let $\tau\in  \QT_2(A)\setminus\{\tau_\infty\},$ then 
$I:= \{x\in A\otimes \K:\tau(x^*x)<\infty\}\neq\{0\}$ is an algebraic ideal of $A\otimes\K.$ 
$A\otimes\K$ is simple implies $I$ is dense in $A\otimes\K.$ 
Then $\Ped(A\otimes\K)\subset I.$ 
\end{proof} 

\begin{prop}\label{shier25-1}
Let $\tau\in\TQT(A)$ and let $a\in (A\otimes\K)_+.$ 
Let $\{e_\lambda\}=\{e\in\Ped(A\otimes\K)_+:\|e\|<1\}$ be an approximate unit of $A.$ 
Then $\lim_\lambda\tau(a^{1/2}e_\lambda a^{1/2})=\lim_{\ep\to 0}\tau((a-\ep)_+).$
\end{prop}
\begin{proof}
Since $a^{1/2}f_{\ep/2}(a)a^{1/2}\ge (a-\ep)_+,$ 
we have 
$\lim_\lambda\tau(a^{1/2}e_\lambda a^{1/2})\ge \lim_{\ep\to 0}\tau((a-\ep)_+).$ Since $\tau$ is lower semicontinuous on $\Ped(A\otimes\K)_+$ {{(see Proposition \ref{jiu25-3}),}} 
we also  have $\tau(a^{1/2}e_\lambda a^{1/2})=\tau(e_\lambda^{1/2} ae_\lambda^{1/2})=\lim_{\ep\to 0}\tau(e_\lambda^{1/2} (a-\ep)_+e_\lambda^{1/2})\le \lim_{\ep\to 0}\tau((a-\ep)_+).$
\end{proof}

\begin{df}\label{shier20-1}
Let $X,Y$ be topological spaces and $X,Y$ are partially ordered sets. 
A homeomorphism $\phi: X\to Y$ is said to be  ordered, if both $\phi$ and $\phi^{-1}$ are order preserving. 
\end{df}

\begin{prop}\label{shier22-3}
Let $A$ be a simple \CA.  
For $\tau\in \TQT(A),$ define $\wtilde \tau:(A\otimes\K)_+\to[0,\infty],$ 
$a\mapsto \lim_{\ep\to 0}\tau((a-\ep)_+).$
Then the map 
$\chi:\tau\mapsto \wtilde \tau$ 
is an 
ordered affine homeomorphism  from 
$\TQT(A)$ to $\QT_2(A)\setminus\{\tau_\infty\}.$
\end{prop} 
\begin{proof}
Using the notations in Proposition \ref{shier25-1},
we have $\wtilde\tau(a)=\lim_\lambda\tau(a^{1/2}e_\lambda a^{1/2})$
for all  $\tau\in\TQT(A)$ and all  $a\in (A\otimes\K)_+.$ 
By \cite[Proposition 2.7]{L22}, 
$\widetilde \tau\in \QT_2(A)\setminus\{\tau_\infty\}$
and $\widetilde \tau$ extends $\tau|_{\Ped(A\otimes\K)_+},$ which imply $\chi$ is injective. 
Let $\tau\in  \QT_2(A)\setminus\{\tau_\infty\},$ then $\tau$ is densely finite by Proposition \ref{jiu27-3}. 
We can extend the map 
$\tau|_{\Ped(A\otimes\K)_+}$
to a map  $\bar  \tau:\Ped(A\otimes\K)\to\C$
canonically, 
and $\bar   \tau\in \TQT(A).$
Since $\tau$ is lower semicontinuous, 
$\chi(\bar   \tau)(a)
=\lim_{\ep\to 0}\bar   \tau((a-\ep)_+)=\tau(a)$ for all $a\in(A\otimes\K)_+.$ 
Hence $\chi(\bar   \tau)=\tau$ and 
$\chi$ is surjective. 
Let $\{\tau_i\}\subset \TQT(A)$ and $\tau\in\TQT(A).$
Then $\lim_i\tau_i=\tau\Leftrightarrow$ $\tau_i$ converges to $\tau$ point-wisely on $\Ped(A\otimes\K)_+\Leftrightarrow$ 
$\chi(\tau_i)$ converges to $\chi(\tau)$ point-wisely on $\Ped(A\otimes\K)_+\Leftrightarrow$  $\lim_i\chi(\tau_i)=\tau$ (by Proposition \ref{jiu29-2} (ii) in the Appendix).
Hence $\chi$ is a homeomorphism. 
It is routine to check that $\chi$ is affine and ordered. 
\end{proof}

\begin{rem} 
By above proposition, {{for a simple \CA\ $A,$}} 
we will identify $\TQT(A)$ with $\QT_2(A)\setminus\{\tau_\infty\}$ 
whenever it is convenience. 
In particular, 
for $\tau\in \TQT(A)$ and $a\in (A\otimes\K)_+\setminus\Ped(A\otimes\K)_+,$ 
$\tau(a):=\lim_{\ep\to 0}\tau((a-\ep)_+).$
\end{rem}

\begin{df}
For $\tau\in \QT_2(A)\cup \TQT(A),$ 
we define 
$d_\tau\in {\rm F}(\Cu(A))$ by 
$d_\tau([a]):=\lim_{n}\tau(f_{1/n}(a))$
for all $a\in (A\otimes\K)_+. 
$
We also write $d_\tau(a)$ for $d_\tau([a])$ when it is convenient.  
Define $\Delta: \QT_2(A)\cup \TQT(A)\to  {\rm F}(\Cu(A)),$
$\tau\mapsto d_\tau.$
\end{df}

\begin{thm}\label{jiu27-4}
{\rm (\cite[Theorem 4.4]{ERS})} 
Both $\QT_2(A)$ and ${\rm F}(\Cu(A))$ are compact and Hausdorff. 
$\Delta: \QT_2(A)\to {\rm F}(\Cu(A)),$ $\tau\mapsto d_\tau$
is an ordered 
affine homeomorphism.  
\end{thm}

We summarize Proposition \ref{shier22-3} 
and Theorem \ref{jiu27-4} in the following form: 

\begin{prop}\label{jiu25-T1}
Let $A$ be a simple  \CA. 
Then the following diagram is commute, and all the maps are 
ordered affine homeomorphisms. 
$$
\xymatrix{
        \TQT(A) \ar[dr]_{\Delta} \ar[rr]_{\chi} & & \QT_2(A)\setminus\{\tau_\infty\} \ar[dl]^{\Delta}\\
         &  {\rm F}(\Cu(A))\setminus\{d_{\tau_{\infty}}\}& 
    }.
$$
\end{prop}

\section{Tracial Oscillation and Real Rank Zero}

We identify $A$ with $A\otimes e_{1,1}\subset A\otimes\K.$
Note that if $A=\Ped(A)$, then $A\otimes e_{1,1}\subset \Ped(A\otimes\K).$ 

\begin{df}
Let $A$ be a \CA. 
For $\tau\in \TQT(A)$ (defined on $\Ped(A\otimes\K)$), 
define 
$\|\tau|_A\|:=\sup\{\tau(a\otimes e_{1,1}) :a\in \Ped(A)_+^1\}.$ Define 
$\QT(A):= 
\left\{\tau\in \TQT(A): \|\tau|_A\|=1\right\}.$ 
Let $\ol{\QT(A)}^w$ 
be the closure of $\QT(A)$ in $\TQT(A).$ 
\end{df}

\begin{prop}{\rm (\cite[Proposition 2.9]{FLosc}, \cite[Lemma 4.5]{eglnp}, \cite[Theorem 4.4]{ERS})}
Let $A$ be an algebraically simple  \CA\ 
 with $\QT(A)\neq\emptyset.$ 
Then $\ol{\QT(A)}^w$ is compact and Hausdorff. 
Moreover, $0\notin \ol{\QT(A)}^w.$
\end{prop}

\begin{df}\label{D2norm}
Let $A$ be a \CA\ with 
$\QT(A)\neq\emptyset.$
For each 
$x\in (A\otimes\K)_+,$ 
define
$\|x\|_{{2}}=\sup\{\tau(x^*x)^{1/2}: \tau\in \ol{\QT(A)}^w\}\in  [0,\infty].$
Let $l^\infty(A)$ be the \CA\ of all norm bounded sequences of $A.$ Define 
$
J_A:=\{\{x_n\}\in l^\infty(A): \lim_{n\to\infty}\|x_n\|_{2}=0\}.
$
\end{df} 

Note that if $a\in A_+,$ then $\|a\|_2\le \|a\|.$ However, when $a\in (A\otimes\K)_+,$ it could happen  that $\|a\|_2> \|a\|.$

\begin{df}\label{DefOS1}
(\cite[Definition A.1]{eglnkk0}, \cite[Definition 4.1]{FLosc})
Let $A$ be an algebraically simple \CA\ with $\QT(A)\neq \emptyset.$ 
Let $a\in  (A\otimes {\cal K})_+,$
define the tracial oscillation of $a$ 
as following: 
\beq
\omega(a):=\lim_{n\to\infty}\sup\left\{d_\tau(a)-\tau(f_{1/n}(a)): \tau\in \ol{\QT(A)}^w\right\}.
\eneq
\end{df}

\begin{rem}
{\bf (1)} In \cite[Definition A.1]{eglnkk0}, $\omega(a)$ is defined by using traces, and the notation used there is $\omega_S(a).$
{\bf (2)} The notation of tracial oscillation used in  \cite[Definition 4.1]{FLosc} is $\omega(a)|_S.$ 
\end{rem}

\begin{df}(\cite[Definition 4.7, Definition 5.1]{FLosc}) 
Let $A$ be an algebraically simple  \CA\ with $\QT(A)\neq \emptyset.$ 
For $a\in {{\Ped(A\otimes\K)_+}},$ $a$ is said to have \emph{tracial approximate oscillation zero}, if for any $\ep>0,$ there is $c\in\Her_A(a)_+$ such that $\|a-c\|_{2}<\ep,$ $\|c\|\le \|a\|,$ and $\omega(c)<\ep.$ 

If $a$ has tracial approximate oscillation zero for all $a\in {{\Ped(A\otimes\K)_+}},$ then the \CA\ $A$ is said to have tracial approximate oscillation zero. 


Let $A$ be an  algebraically simple \CA\ with $\QT(A)\neq \emptyset$  and $A$ has tracial approximate oscillation zero. 
If $B$ is another \CA\ and  $A\otimes\K\cong B\otimes\K,$
then $B$ is also said to have tracial approximate oscillation zero. 
\end{df}
\begin{rem}
Above definition was called T-tracial approximate oscillation zero in \cite[Definition 5.1]{FLosc}.  There are also other variations of tracial approximate oscillation in \cite[Definition 4.7]{FLosc}. 

\end{rem}

The following is a collection of some useful properties of tracial oscillation: 

\begin{prop}\label{1204-2}
Let $A$ be an algebraically simple \CA\ with $\QT(A)\neq \emptyset.$

{\rm (i)} {\rm (\cite[Proposition 4.2]{FLosc})} Let $a,b\in \Ped(A\otimes\K)_+.$ If $a\sim b,$ then $\omega(a)=\omega(b).$

{\rm (ii)} {\rm (\cite[Proposition 4.4 (2)]{FLosc})} {{Let $a,b\in \Ped(A\otimes\K)_+.$ If $ab=0,$ then $\omega(a+b)\le \omega(a)+\omega(b).$}}

{\rm (iii)} {\rm (\cite{FLosc}, p577, line 10)} For $a\in\Ped(A\otimes\K)_+,$  $\omega(a)=0$ if and only if the map $\wh{[a]}: \TQT(A)\to \R_+, \tau\mapsto d_\tau(a)$ is continuous. 

{\rm (iv)} {\rm (\cite[Proposition 5.7]{FLosc})} If, in addition,  $A$ is $\sigma$-unital, then $a\in\Ped(A\otimes\K)_+$ has tracial approximate oscillation zero if and only if for any $\ep>0,$ there is $c\in\Her (A)_+$ such that $\|a-c\|<\ep$ and $\omega(c)<\ep.$
\end{prop}

The following is an equivalent description of tracial approximate oscillation zero (for unital simple \CAs). 
\begin{prop}\label{0421-4}
Let $A$ be a unital 
simple \CA\ with $\QT(A)\neq \emptyset.$ 
The following are equivalent: 

{\rm (i)} For all $a\in A_+,$ $a$ has tracial approximate oscillation zero. 

{\rm (ii)} For any $x\in A_{sa}^1$ and $\ep>0,$ there are $y\in A_{sa}$ with $\|y\|\le 1$ and $\dt>0,$ such that $\|x-y\|_2<\ep$ and 
$\sup\{\tau(g_\dt(y)):\tau\in \QT(A)\}<\ep,$
where $g_\dt(t):=1-f_\dt(|t|)$ for all $t\in\R.$  

\end{prop}
\begin{proof}
${\rm (i)\Rightarrow (ii):}$ 
Let $x\in A^1_{sa}\nzero$ and let $\ep\in(0,1).$ 
Since $x_+,x_-$ has tracial approximate oscillation zero, 
there are $a\in \Her(x_+)_+$ and $b\in \Her(x_-)_+$ such that 
$\|a\|\le  1,$ $\|b\|\le 1,$ 
$\omega(a),\omega(b)<\ep/10,$ 
and $\|x_+-a\|_2,\|x_--b\|_2$ are sufficiently small, 
such that $\|x-(a-b)\|_2=\|(x_+-a)-(x_--b)\|_2<\ep/2.$
Since $\omega(a),\omega(b)< \ep/10,$ there is $\theta>0$ such that 
\beq\label{shi25-2}
\hspace{-0.2in} 
\sup\{d_\tau(a)-\tau(f_{3\theta}(a)):\tau\in \QT(A)\}, 
\sup\{d_\tau(b)-\tau(f_{3\theta}(b)):\tau\in \QT(A)\}<\ep/10 .
\eneq
Let $h$ be a real valued continuous function on $\R$ such that $h(t)=0$ for $t\in[-2\theta,2\theta]$ and $|h(t)-t|\le 2\theta$  for all $t\in\R.$ 
Then $h(t)\perp g_\theta(t).$ 
Let $s(t):={\theta} \cdot g_{_\theta}(t)+h(t).$ 
Let  $y:=s(a-b).$ 
Then $\|y-(a-b)\|<2\theta.$ With $\theta$ sufficiently small, we have 
\beq\label{shier23-2}
\|y-x\|_2=
\|(y-(a-b))-(x-(a-b))\|_2<\ep. 
\eneq
Let $\dt>0$ be sufficiently small, such that 
$f_\dt(|s(t)|)=1$ for $t\in (-\infty,-3\theta]\cup[-\theta,\theta]\cup[3\theta,\infty).$ 
Then $1\le  f_\dt(|s(t)|)+f_{\theta/2}(|t|)-f_{3\theta}(|t|)$ for all $t\in\R.$ 
Thus $g_\dt(y)=1- f_\dt(|s(a-b)|)\le f_{\theta/2}(|a-b|)-f_{3\theta}(|a-b|)
=f_{\theta/2}(a)-f_{3\theta}(a)+f_{\theta/2}(b)-f_{3\theta}(b).$ 
Then for all $\tau\in\QT(A),$ 
\beq
\tau(g_\dt(y))&\le & \tau(f_{\theta/2}(a)-f_{3\theta}(a)+f_{\theta/2}(b)-f_{3\theta}(b))
\\\label{shier23-3}
&=&
\tau(f_{\theta/2}(a))-\tau(f_{3\theta}(a))+\tau(f_{\theta/2}(b))-\tau(f_{3\theta}(b)))\overset{\eqref{shi25-2}}{\le} \ep.
\eneq
Hence \eqref{shier23-2} and \eqref{shier23-3} show that (ii) holds.

${\rm (ii)\Rightarrow (i):}$ 
Let $a\in A_+^1$ and 
let $\ep\in(0,1).$ Let {{$\theta:=\ep^4/100.$}}
Let {{$a_0:=a-\theta.$}} 
Define 
$h(t):=f_\theta(t+\theta)$ for $t\in\R.$  
Let $\eta\in  (0,\theta).$
By (ii),
there is $x\in A_{sa}^1$ and $\dt\in(0,\eta)$ such that  
\beq\label{shier23-f5}
\hspace{-0.2in}
\|a_0-x\|_2<\eta 
\quad \mbox{and} \quad 
\sup\{\tau(g_\dt(x)):\tau\in \QT(A)\}<\eta.
\eneq 
We may assume that $\eta$ is sufficiently small such that 
 $\|a_0-x\|_2<\eta$ imply the following:  
\beq
\hspace{-0.2in}
&&\|a-(x+\theta)_+\|_2=\|(a_0+\theta)_+-(x+\theta)_+\|_2<\theta, \label{shier23-f1}
\\&& \|h(a_0)b h(a_0)-h(x)b h(x)\|_2<\theta \rforal b \in A_+^1,\mbox{\ and}\label{shier23-f2}
\\&&\|a^{1/2}b a^{1/2}-(x+\theta)_+^{1/2}b(x+\theta)_+^{1/2}\|_2<\theta 
 \rforal b \in A_+^1.\label{shier23-f3}
 \eneq
Let $x_0:=f_{\dt/4}(x),$ $x_1:=f_{\dt/8}(x),$ 
$e:=h(a_0)x_0h(a_0).$ 
Then $h(x)x_0=x_0.$ For any $\tau\in\QT(A),$ 
\beq
\tau(e)=\tau(h(a_0)x_0h(a_0))
\overset{\eqref{shier23-f2}}{\ge} \tau(h(x)x_0h(x))-2\theta=\tau(x_0)-2\theta.
\eneq
Note that $d_\tau(e)=d_\tau(h(a_0)x_0h(a_0))\le d_\tau(x_0)\le \tau(x_1).$
Then 
\beq
d_\tau(e)-\tau(e)
\le \tau(x_1)-\tau(x_0)+2\theta\le \tau(g_\dt(x))+2\theta\le 3\theta
\le \ep. 
\eneq 
Hence $\omega(e)\le \ep.$ 
Let $a_1:=a^{1/2}ea^{1/2}\le a.$
Since $e\in\Her(a),$ we have $a_1\sim e.$ Then $\omega(a_1)=\omega(e)\le \ep$ 
(see Proposition \ref{1204-2} (i)). 
For all $\tau\in\QT(A),$
\beq
\tau(a-a_1)&=&\tau(a-a^{1/2}h(a_0)x_0h(a_0)a^{1/2})
 \le  
\tau(a-a^{1/2}x_0a^{1/2})+2\sqrt{\theta}
\qquad \\&\overset{\eqref{shier23-f1}}{\le} &
\tau((x+\theta)_+-(x+\theta)_+^{1/2}x_0(x+\theta)_+^{1/2})+2\sqrt{\theta}+3\theta 
\\&
{\le} &
(\theta+\dt) + 2\sqrt{\theta}+3\theta \le \ep^2.
\eneq
Then $\tau((a-a_1)^2)^{1/2}\le \tau(a-a_1)^{1/2}\le \ep.$
Hence $\|a-a_1\|_2\le \ep.$ Then (i) holds. 
\end{proof}

\begin{rem} 
Recently, a notion called Property (S) was introduced in \cite{EN}, which is similar to tracial approximate oscillation zero. Property (S) is the condition (ii) in Proposition \ref{0421-4} without assuming $\|y\|\le 1.$ 
\end{rem}

\begin{lem}\label{shier14-3}
Let $A$ be a \CA\ with $\QT(A)\neq\emptyset.$
Let $a\in \Ped(A\otimes\K)_+$ and $\ep>0.$ Then there are $b\in\Her(a)_+$ with $b\le a,$ $n\in\N,$ and injective *-homomorphism $\phi:\Her(b)\to M_n(A)$ such that $\|a-b\|_2<\ep,$ $\|\phi(x)\|_2=\|x\|_2,$ and $\omega(\phi(x))=\omega(x)$for all $x\in\Her(b)_+.$ 
\end{lem}
\begin{proof} 
By Proposition \ref{jiu25-3}, we have $\lim_k\tau(a-(a-1/k)_+)=0.$ Since $\ol{\QT(A)}^w$ is compact and $\{\tau(a-(a-1/k)_+)\}_k$ is decreasing for all $\tau$, by Dini's theorem, there is $m\in\N$
such that $\sup\{\tau(a-(a-1/m)_+):\tau\in \ol{\QT(A)}^w\}<(\ep/4)^2.$ 
Then for all $\tau\in \ol{\QT(A)}^w,$ 
$\tau((a-(a-1/m)_+)^2)\le \tau(a-(a-1/m)_+)<(\ep/4)^2.$
Hence $\|a-(a-1/m)_+\|_2\le \ep/4.$ 
Let $b:=(a-1/m)_+\in \Her(a).$

There are $n\in \N$ 
and $c\in M_n(A)_+$ such that  $\|a-c\|<1/4m.$ 
By \cite[Proposition 2.4 (iv)]{RorUHF2}, there is $r\in A\otimes\K$ such that $(a-1/m)_+=r^* r$ and $rr^*\in \Her(c)\subset M_n(A).$ 
Let $r=v|r|$ be the polar decomposition of $r$ in $A^{**}.$ 
Then the map $\phi:\Her(b)=\Her(r^*r)\to \Her(rr^*)\subset M_n(A),$
$x\mapsto vxv^*$
is an injective *-homomorphism. Moreover, for any $x\in \Her(b)_+,$
$vx^{1/2}\in A.$ Then $\tau(\phi(x))=\tau(vxv^*)=\tau(x^{1/2}v^*vx^{1/2})=\tau(x)$ for all $\tau,$ which implies $\omega(\phi(x))=\omega(x)$ and  $\|\tau(\phi(x))\|_2=\|\tau(x)\|_2$ for all $x\in\Her(b)_+.$ 
\end{proof}



Recall that a \CA\ $A$ is said to have real rank zero, if the set of invertible self-adjoint elements in $\wtilde A_{sa}$ is dense in $\wtilde A_{sa}$ (\cite{BP91}). 

\begin{thm}\label{1204-6}
Let $A$ be {{a separable 
simple}} \CA\ with $\QT(A)\neq \emptyset.$ 
The following are equivalent: 

{\rm (i)} $A$ has tracial approximate oscillation zero, i.e., 
for all $a\in \Ped(A\otimes\K)_+,$ $a$ has tracial approximate oscillation zero. 

{\rm (ii)} For all $a\in {{\Ped(A)_+}},$ $a$ has tracial approximate oscillation zero. 

{\rm (iii)} 
{{$l^\infty(A)/J_A$}} has real rank zero. 
\end{thm}
\begin{proof}
${\rm (i)\Rightarrow (ii)}$ is trivial. 
Note that ${\rm (ii)\Rightarrow (iii)}$ was proved in \cite[Theorem 6.4]{FLosc}. 

${\rm (iii)\Rightarrow (ii):}$ 
Denote $l^\infty(A)/
J_A$ by $B.$
Let $\pi: l^\infty(A)\to B$ be the quotient map. 
We may embed $A$ into $l^\infty(A)$ canonically. 
Let $a\in A_+^1$ and let $\ep>0.$  
 Since $B$ has real rank zero, 
 by \cite[2.6 (iii)]{BP91}, 
there is a projection $p\in \Her_B(\pi(a))\subset 
\pi(l^\infty(\Her_A(a)))$ such that 
\beq\label{shi27-1}
\|\pi(a)-p\pi(a)\|<\ep/2.
\eneq
Let $p=\pi(\{e_n\})$ with $e_n\in \Her_A(a)$ for all $n.$ 
Then $\{f_{1/4}(e_n)\}$ is a permanent projection lifting of $p$ (see \cite[Definition 6.1, Proposition 6.2(1)]{FLosc}), and thus 
$\lim_n\omega(e_n)=0$
(see \cite[Definition 6.2(2)]{FLosc}).  
Together with \eqref{shi27-1}, 
there is $m\in\N$ such that 
\beq
\|a-f_{1/4}(e_m)a\|_2<\ep, 
\quad\mbox{ and }
\quad
\omega(e_m)<\ep. 
\eneq
Hence, by \cite[Proposition 5.6 (2)]{FLosc}, $a$ has tracial approximate oscillation zero. 

Note that we have shown (ii) equivalents to (iii). 

{{${\rm (ii)\Rightarrow (i)}$:}}
Let $a\in \Ped(A\otimes\K)_+.$ Let $\ep>0.$ 
By Lemma \ref{shier14-3}, there are $n\in\N,$ $b\in \Her(a)_+$ with $b\le a$ and an injective *-homomorphism $\phi: \Her(b)\to M_n(A)$ 
such that $\|a-b\|_2<\ep/8,$ 
$\|\phi(x)\|_2=\|x\|_2,$ and $\omega(\phi(x))=\omega(x)$for all $x\in\Her(b)_+.$ 
 Since (ii) holds and (ii) implies (iii) (see \cite[Theorem 6.4]{FLosc}), 
$l^\infty(A)/J_A$ has real rank zero. Then $l^\infty(M_n(A))/J_{M_n(A)}\cong M_n(l^\infty(A)/J_A)$ also has real rank zero. 
By what we just proved (i.e., ${\rm (iii)\Rightarrow (ii)}$), we have that $\phi(b)$ has tracial approximate oscillation zero. 
Then there is $c\in \Her(\phi(b))_+$ such that  $\omega(c)<\ep$  and $\|\phi(b)-c\|_{2}<\ep/8.$ 
Let $d:=\phi^{-1}(c)\in\Her(b),$ then 
$\omega(d)=\omega(c)<\ep$ and $\|b-d\|_2=\|\phi^{-1}(\phi(b)-c)\|_2
=\|\phi(\phi^{-1}(\phi(b)-c))\|_2=\|\phi(b)-c\|_{2} <\ep/8.$ 
Then by \cite[Lemma 3.5 (2)]{Haa}, we have 
$\|a-d\|_2^{2/3}\le \|a-b\|_2^{2/3}+ \|b-d\|_2^{2/3}\le 2(\ep/8)^{2/3}.$
Then $\|a-d\|_2<\ep.$
Then $a\in \Ped(A\otimes\K)_+$ has tracial approximate oscillation zero and (i) holds. 
\end{proof}

\section{Hereditary Surjective Canonical Map}


\begin{df}
(\cite[5.1]{ERS})
Let $A$ be a \CA. 
Let $\Lsc({\rm F}(\Cu(A)))$ be the set of functions 
$f:{\rm F}(\Cu(A))\to[0,\infty]$ that are additive, order-preserving, homogeneous (with respect to $\R_+$),
lower semicontinuous,  
and satisfies $f(0)=0.$ 

For  $f,g\in \Lsc({\rm F}(\Cu(A))),$
if $f\le (1-\ep)g$ for some $\ep>0$ and  
$f$ is continuous at each $\tau\in {\rm F}(\Cu(A))$
satisfying $g(\tau)<\infty,$
then we write $f\triangleleft g.$

Denote by ${\rm L}({\rm F}(\Cu(A)))$ the sub-semigroup of $\Lsc ({\rm F}(\Cu(A)))$
consisting of those $f\in \Lsc({\rm F}(\Cu(A)))$ that can be written as 
the pointwise supremum of a sequence $\{f_n\}_{n\in\N}$ in 
$\Lsc({\rm F}(\Cu(A)))$ such that $f_n\triangleleft f_{n+1}$ for all $n\in\N.$ 
\end{df}

\begin{df}\label{DGamma}
Let $A$ be a \CA\ with $\TQT(A)\neq\{0\}.$
Denote by  $\Aff\left(\TQT(A)\right)$ the set of continuous real valued 
functions $f$  on $\TQT(A)$ 
such that $f(s \tau)=s f(\tau),$ 
$f(\tau+\sigma)=f(\tau)+f(\sigma)$
for all $s\in \R_+$ and $\tau,\sigma\in \TQT(A).$ 
{{Note that if $f\in \Aff\left(\TQT(A)\right),$ then $f(0)=0.$}} 
Moreover, define 
\beq
\hspace{-0.2in}
&\Aff_+\left(\TQT(A)\right):=\left\{f\in \Aff\left(\TQT(A)\right):  f(\tau)>0\mbox{ if }\tau\in \TQT(A)\setminus \{0\}\right\}\cup \{0\},&
 \nonumber\\
&{\rm LAff}_+\left(\TQT(A)\right):=
\left\{f:\TQT(A)\to [0,\infty]: \exists\ \{f_n\}\subset  \Aff_+\left(\TQT(A)\right)\mbox{ with } f_n\nearrow f\right\},&
\nonumber
\eneq 
where $f_n\nearrow f$ means for all $n\in\N$ and all $\tau,$ $f_n(\tau)\le f_{n+1}(\tau),$  and $f(\tau)=\lim_i f_i(\tau).$ 
\end{df}


\begin{df}\label{shi28-1}
Let $A$ be a \CA. 
For $a\in {{(A\otimes\K)_+}},$ define a map 
$\wh{[a]}:\TQT(A)\cup \QT_2(A)\to [0, \infty],$ 
$\tau\mapsto d_\tau(a),$ also define a map 
$\wh{a}: {{\TQT(A)\cup \QT_2(A)\to [0, \infty]}},$ $\tau\to \lim_{n\to\infty}\tau((a-1/n)_+).$ 
The  canonical map $\Gamma$ is defined as following: 
\beq
\Gamma: \Cu(A)\to {\rm LAff}_+\left(\TQT(A)\right),\quad 
[a]\mapsto \wh{[a]}.
\eneq
\end{df}

\begin{prop}\label{shi21-4}
Let $A$ be a simple \CA\ with $\TQT(A)\neq \{0\}.$ Define a map $\Delta: \TQT(A)\to {\rm F}(\Cu(A))$ by 
$\tau\mapsto d_\tau.$ Then for any $f\in {\rm L}({\rm F}(\Cu(A))),$ $f\circ\Delta\in \LAff_+\left(\TQT(A)\right),$ and the map $\Delta^*: {\rm L}({\rm F}(\Cu(A)))\to \LAff_+\left(\TQT(A)\right),$ $f\mapsto f\circ\Delta$ is an affine bijection. 
\end{prop}
\begin{proof}
By \cite[Theorem 5.7]{ERS}, there is a sequence $\{a_n\}\subset (A\otimes\K)_+$ such that for all  $\tau\in \TQT(A),$
$f\circ\Delta(\tau)=f(d_\tau)=\lim_{n} \tau(a_n)$ and  
 $\tau(a_n)\le \tau(a_{n+1})$ for all $n.$ 
Moreover, by the proof of \cite[Theorem 5.7]{ERS}, $a_n$ can be  chosen in $\Ped(A\otimes\K)_+$ for all $n.$ 
Then $\wh {a_n}\in \Aff_+\left(\TQT(A)\right)$ for all $n.$ 
Then $\{\wh{a_n}\}$ is an increasing sequence and 
$f\circ\Delta=\lim_n \wh{a_n}.$ Hence $f\circ\Delta\in \LAff_+\left(\TQT(A)\right).$ 

$\Delta^*$ is affine because $\Delta$ is  affine.
{{By Proposition \ref{jiu25-T1},}}
${\rm F}(\Cu(A))=\{d_\tau:\tau\in \TQT(A)\}\cup\{d_{\tau_\infty}\}.$
Let $f_1,f_2\in {\rm L}({\rm F}(\Cu(A)))$ such that $\Delta^*(f_1)=\Delta^*(f_2).$
Then $f_1(d_\tau)=f_2(d_\tau)$ for all $\tau\in \TQT(A).$ 
Let $\tau\in \TQT(A)\nzero,$ then  
$\tau_\infty=\lim_n n\tau.$ 
Since $f_1,f_2$ are lower semicontinuous, 
$f_1(d_{\tau_\infty})=\lim_nf_1(d_{n\tau})=\lim_n f_2(d_{n\tau})=f_2(d_{\tau_\infty}).$
Thus $f_1=f_2.$ Hence $\Delta^*$ is injective. 

Let $g\in \LAff_+(\TQT(A)).$ Then there is an increasing sequence $\{g_n\}\subset \Aff_+(\TQT(A))$ such that $g=\lim_n g_n.$ 
For each $n\in\N,$ define $\bar g_n: {\rm F}(\Cu(A))\to[0,\infty]$ by $\bar g_n(d_\tau)=g_n(\tau)$ for $\tau\in \TQT(A)$ and $\bar g_n(d_{\tau_\infty})=\infty.$ 
Then $\bar g_n$ is an increasing sequence of continuous affine functions on ${\rm F}(\Cu(A))$. 
Since the map $d_\tau\mapsto \tau$ is order preserving, $\bar g_n$ is also {{order preserving}}. Hence $\bar g_n\in \Lsc({\rm F}(\Cu(A))).$
Let $\bar g:=\lim_n\bar g_n=\lim_n (1-1/n)\bar g_n\in {\rm L}({\rm F}(\Cu(A))).$ Then $g=\bar g\circ\Delta.$ Hence $\Delta$ is surjective. 
\end{proof}

\begin{df}\label{HER-surj}
Let $A$ be a \CA\ with $\TQT(A)\neq\{0\}.$ 
If for any $a\in (A\otimes\K)_+$ and any $f\in \LAff_+\left(\TQT(A)\right)$
satisfying $f(\tau) <d_\tau(a)$ for all $\tau\in \TQT(A)\nzero,$ 
there is $b\in (A\otimes\K)_+$ such that 
$b\lesssim a$ and $d_\tau(b)=f(\tau)$ for all $\tau\in \TQT(A),$
then we say the 
canonical map $\Gamma: \Cu(A)\to \LAff_+\left(\TQT(A)\right)$
is {{hereditary surjective}}. 
\end{df}

Recall that a simple \CA\ $A$ with $\TQT(A)\neq\{0\}$ is said to have 
strict comparison for positive elements, 
if for any $a,b\in \Ped(A\otimes\K)_+,$
$d_\tau(a)<d_\tau(b)$ for all $\tau\in \TQT(A)\nzero$ implies $a\lesssim b.$ 

\begin{thm}
Let $A$ be a simple \CA\ with $\TQT(A)\neq\emptyset,$ has strict comparison for positive elements, and $\Gamma$ is surjective.  Then $\Gamma$ is hereditary surjective. 
\end{thm}
\begin{proof}
Let $a\in (A\otimes\K)_+$ and let $f\in \LAff_+\left(\TQT(A)\right)$
satisfying $f(\tau) <d_\tau(a)$ for all $\tau\in \TQT(A)\nzero.$
Since $\Gamma$ is surjective, there is $b\in (A\otimes\K)_+$ such that $\Gamma([b])=f.$ Then  
$d_\tau(b)=\Gamma([b])(\tau)=f(\tau)<d_\tau(a)$ for all $\tau\neq 0.$ 
Then $b\lesssim a$ due to strict comparison. 
\end{proof} 

Recall that a \CA\ $A$ is said to have stable rank one, if 
the set of invertible elements of $\wtd A$ is dense 
in $\wtd A$ 
(\cite{Rff}).    
The following result is a  consequence of Antoine-Perera-Robert-Thiel's results (\cite{APRT}). 
In the following theorem, 
for $x\in \Cu(A),$ 
we use $\wh x$ to denote the evaluation map 
$\wh x: {\rm F}(\Cu(A))\to \R_+, \lambda\mapsto \lambda(x).$
Note that  $\wh x\in {\rm L}({\rm F}(\Cu(A))).$ 

\begin{thm}\label{jiu21-a1}
{\rm (cf.~\cite[Theorem 7.14]{APRT})}
Let $A$ be a simple separable \CA\ of stable rank one.  
Let $a\in\Cu(A)$ and let $f\in {\rm L}({\rm F}(\Cu(A)))$
satisfying $f\le \wh a.$ 
Then there is  
$c\in\Cu(A)$ such that 
$c\le a$ and $\wh c=f.$
\end{thm}
\begin{proof}
By \cite[Theorem 7.13]{APRT}, 
there is $b\in\Cu(A)$ such that $\wh{b}=f.$ 
By \cite[Theorem 3.8]{APRT},  the infimum $b\wedge a$ of 
$\{b, a\}$ exists. 
Let $c=b\wedge a\in \Cu(A)$, then $c\le a.$
By \cite[Theorem 6.12]{APRT} and $f\le \wh a$, we have
$
\wh c=\wh{b\wedge a}=\wh{b}\wedge \wh a
=f\wedge \wh a=f. 
$
\end{proof}

\begin{thm}\label{shi21-3}
Let $A$ be a simple separable \CA\ of stable rank one. 
Then the canonical map $\Gamma: \Cu(A)\to \LAff_+\left(\TQT(A)\right)$
is hereditary surjective.
\end{thm}
\begin{proof}
Let $a\in (A\otimes\K)_+$ and let $f\in \LAff_+\left(\TQT(A)\right)$
satisfying $f(\tau)< d_\tau(a)=\wh{[a]}(\tau)$ for all $\tau\in\TQT(A)\nzero.$ 
Let $\Delta, \Delta^*$ be 
as in Proposition \ref{shi21-4}. 
Then by Proposition \ref{shi21-4}, there is $g\in {\rm L}({\rm F}(\Cu(A)))$ such that $f=\Delta^*(g)=g\circ\Delta.$
For any $\tau\in\TQT(A)\nzero,$ $g(d_\tau)=g\circ\Delta(\tau)=f(\tau)<d_\tau(a).$ Hence $g\le \wh{[a]}.$ 
By Theorem \ref{jiu21-a1}, there is $b\in (A\otimes\K)_+$ such that $b\lesssim a$ and $\wh{[b]}=g\in  {\rm L}({\rm F}(\Cu(A))).$ 
Then 
for all $\tau\in \TQT(A),$ 
$d_\tau(b)=\wh{[b]}(d_\tau)=g(d_\tau)=g\circ\Delta(\tau)=f(\tau).$ 
Thus $\Gamma$ is hereditary surjective. 
\end{proof} 

\section{Hereditary Surjective Implies Tracial Approximate Oscillation Zero}

\begin{lem}\label{jiu15-1}
{\rm (\cite[Lemma 4.5]{FLosc})}
Let $a\in (A\otimes \K)_+$ and let $\lambda>0.$
Assume that $\omega(a)<\lambda.$ Then there exists $\dt_0>0$
such that for all $\dt\in (0,\dt_0),$ $\omega(f_\dt(a))<\lambda.$ 
\end{lem}

\begin{prop}\label{jiu23-2}
Let $A$ be a \CA\ and let $\tau$ be a quasitrace on $A.$
Then for any $x\in A$ and any $f\in C_0((0,+\infty))_+,$
$\tau(f(x^*x))=\tau(f(xx^*)).$
\end{prop}
\begin{proof}
Let $x=u|x|$ be the polar decomposition of $x$  in  $A^{**}.$
Then $y:=uf(x^*x)^{1/2}\in A$ (see \cite[1.3]{Cuntz77}), 
and $yy^*=f(xx^*).$
Thus  $\tau(f(x^*x))=\tau(y^*y)=\tau(yy^*)=\tau(f(xx^*)).$
\end{proof}

{{The following lemma is one version of Lin's orthogonality  lemma,}} which 
was essentially contained in the proof of \cite[Lemma 3.1]{LinJFA} and \cite[Proposition 2.30]{LinAdv}. 
We include a proof here for reader's convenience. 


\begin{lem}\label{jiu16-prop1}
{\rm (Lin's orthogonality  lemma, cf.~\cite[Lemma 3.1]{LinJFA}, \cite[Proposition 2.30]{LinAdv})}
Let $A$ be an 
algebraically  simple
\CA. Let $a\in (A\otimes\K)_+,$
$e\in \Ped(A\otimes \K)_+,$ and $\ep>0.$
Assume that $a\lesssim e$
and $\omega(a)<\ep.$  
Then 
there are $b,c\in \Her(e)_+^1$ and 
 such that   

(1) $d_\tau(a)-\ep<
\tau(b)\le d_\tau(b) \le d_\tau(a)$ for all 
$\tau\in \overline{\QT(A)}^w;$  

(2) $\omega(b)<\ep;$

(3) $bc=0;$

(4) 
$d_\tau(e)-d_\tau(a)\le d_\tau(c)<d_\tau(e)-d_\tau(a)+\ep
\label{jiu16-4}
$ for all $\tau\in  \overline{\QT(A)}^w.$
\end{lem}

\begin{proof}
Since $\omega(a)<\ep,$ there exists $\theta_0>0$ such that 
for any $\theta_1\in (0,\theta_0),$
\beq\label{jiu16-11}
\sup_{\tau\in  \overline{\QT(A)}^w}
\{d_\tau(a)-\tau(f_{\theta_1}(a))\}<\ep.
\eneq
By Lemma \ref{jiu15-1}, there exists 
$\theta\in(0,\theta_0/4)$ such that 
$
\omega(f_{\theta}(a))<\ep. 
$
Since $a\lesssim e,$ by \cite[Proposition 2.4 (iv)]{RorUHF2}, there are 
$\eta>0$ 
and  $x\in A\otimes \K$ such that $x^*x=f_\theta(a),$ and $xx^* \in \Her(f_\eta(e)).$ 
Note that by \cite[Proposition 4.2]{FLosc}, 
$\omega(xx^*)=\omega(x^*x)=\omega(f_\theta(a))<\ep.$
By Lemma \ref{jiu15-1}, there is $\dt>0$ such that 
$\omega(f_\dt(xx^*))<\ep.
$
Let $b_0:=f_{\dt/4}(xx^*),$
$b_1:=f_{\dt/2}(xx^*),$
and $b:=f_\dt(xx^*).$  
Then $\omega(b)<\ep,$
thus (2) holds. 

Since $b\lesssim xx^*\sim x^*x=f_\theta(a)\lesssim a,$ 
$\tau(b)\le d_\tau(b)\le d_\tau(a)$ for all $\tau.$
We also have 
\beq
d_\tau(a)-\tau(b)&=&
d_\tau(a)-\tau(f_\dt(xx^*))
 \overset{\rm{(Prop\  \ref{jiu23-2})}}{=} 
d_\tau(a)-\tau(f_\dt(x^*x))
\\&= &
d_\tau(a)-\tau(f_\dt(f_\theta(a)))
 \le  
d_\tau(a)-\tau(f_{2\theta}(a))
 \overset{\eqref{jiu16-11}}{<} \ep. \label{jiu16-2}
\eneq
Then \eqref{jiu16-2} shows that (1) holds. 
By continuous function calculus,
there is $g\in C^*(e)_+^1$ such that $g\sim e$ and $g f_\eta(e)=f_\eta(e).$ Define 
$c:=g-b_1\in \Her(e)_+^1.$ 
Since $gb=b$ and $b_1b=b,$  we have 
$bc=0,$ hence (3) holds. 
Note that  $c+b_0=g+(b_0-b_1)\ge g\sim e.$  
For all $\tau\in \overline{\QT(A)}^w$, we have 
\beq\label{jiu16-3}
d_\tau(e)=d_\tau(g)
&\le & d_\tau(c)+d_\tau(b_0)\le 
d_\tau(c)+d_\tau(xx^*)
=
d_\tau(c)+d_\tau(x^*x)
\\
&\le& 
d_\tau(c)+d_\tau(a).
\eneq

For the last step, 
since $c\perp b$ and $b+c=g-(b_1-b)\le g\sim e,$ we have 
\beq\label{jiu16-q5}
d_\tau(c)+ d_\tau(b)
=
d_\tau(c+b)
\le 
d_\tau(g)=d_\tau(e).
\eneq 
Then we have 
\beq\label{jiu16-5}
d_\tau(c)
\overset{\eqref{jiu16-q5}}{\le} d_\tau(e)-d_\tau(b)
\overset{\eqref{jiu16-2}}{<}d_\tau(e)-d_\tau(a)+\ep. 
\eneq
Then \eqref{jiu16-3} and \eqref{jiu16-5} show that 
(4) holds. 
\end{proof}

The following theorem characterizes when does $\Gamma$ being  hereditary surjective. In particular, (2) of the following theorem shows that to verify the hereditary surjectivity of $\Gamma,$
it is suffice to look at continuous affine functions. 
The proof of the following theorem used Lin's technique  (\cite[Lemma 3.3]{LinJFA}). 

\begin{thm}\label{shier09-1}
Let $A$ be an algebraically simple \CA\ with $\QT(A)\neq\emptyset.$ 
Then the following are equivalent: 

{\rm (1)} $\Gamma$ is {{hereditary surjective}}.

{\rm (2)} For any $e\in (A\otimes\K)_+$ and any $h\in \Aff_+\left(\TQT(A)\right)$
satisfying $h(\tau) <d_\tau(e)$ for all $\tau\in \TQT(A)\nzero,$ there is $y\in (A\otimes\K)_+$ such that 
$y\lesssim e$ and $d_\tau(y)=h(\tau)$ for all $\tau\in \TQT(A).$ 

{\rm (3)} For any $e\in (A\otimes\K)_+,$ any $\ep>0,$ and any $h\in \LAff_+\left(\TQT(A)\right)$
satisfying $h(\tau) <d_\tau(e)$ for all $\tau\in \TQT(A)\nzero,$ 
there is $y \in (A\otimes\K)_+$ such that  
$y\lesssim e$ and $d_\tau(y)=h(\tau)$ for all $\tau\in \TQT(A),$
and  $y=\sum_{n=1}^\infty y_n,$ where 
$\{y_n\}_{n\in\N}\subset (A\otimes\K)_+$ is a sequence of mutually orthogonal positive elements  satisfying   $\omega(y_n)=0$ and $\|y_n\|<1/2^n$ for all $n\in\N.$ 
{{Moreover, for any $\ep>0,$ there is a sequence of mutually orthogonal positive elements $\{b_n\}\subset \Her(a)_+$ such that $ \omega(b_n)<\ep/2^n$ for all $n,$ and $|h(\tau)-\sum_{n=1}^\infty d_\tau(b_n)|<\ep$ for all $\tau\in \ol{\QT(A)}^w.$}}
\end{thm}
\begin{proof} 
${\rm (3)\Rightarrow (1)\Rightarrow (2)}$ are trivial. 
In the following we will 
show ${\rm (2)\Rightarrow (3)}.$ 
{{If $A$ is an elementary \CA, then $\TQT(A)$ is generated by one 2-quasitrace, which is a trivial case.}} 
Hence in the following we may assume that $A$ is non-elementary. 
Let $e\in (A\otimes\K)_+$ with $\|e\|=1$ and let $h\in \LAff_+\left(\TQT(A)\right)$ satisfying $h(\tau) <d_\tau(e)$ for all $\tau\in \TQT(A)\nzero.$ 
Let $S:=\ol{\QT(A)}^w.$ 
Let $\ep>0.$ Without loss of generality, we may assume that  
\beq\label{1202-1}
\ep <d_\tau(e)-h(\tau) \rforal  \tau\in S.
\eneq 
{{Note that $\TQT(A)=\R_+\cdot S$ and $S$ is compact. Then by \cite[Proposition 3.2]{LinJFA}, there is a sequence $\{h_i\}\subset \Aff_+\left(\TQT(A)\right)$ such that}}  
\beq\label{shi28-4}
h(\tau)=\sum_{i=1}^\infty h_i(\tau) \rforal \tau\in \TQT(A)\quad(=\R_+\cdot S).
\eneq

Since $A$ is non-elementary, there is $r' \in \Her(f_{1/4}(e))$ such that 
$\sup\{d_\tau(r'): \tau\in S\}<\ep/4.$ By using continuous functional calculus, we can find $\bar r, r\in\Her(r')_+^1$ such that $\bar r\cdot r=r.$ 
Let $g\in C_0((0,1])_+$ be the  function that $g(0)=0,$ $g(t)=1$ for $t\in [1/4,1],$ and $g$ is linear on $[0,1/4].$ Then $e\sim g(e)$ and $g(e)\cdot \bar r=\bar r.$ 
Let $\bar e:=g(e)-\bar r\in \Her(e)_+^1.$ 
Then 
\beq\label{shier10-1}
\bar e \cdot r=(g(e)-\bar r)r=0.
\eneq
For any $\tau\in S,$ 
$d_\tau(e)=d_\tau(g(e))=d_\tau(\bar e+\bar r)\le d_\tau(\bar e)+d_\tau(\bar r)<d_\tau(\bar e)+\ep.$ 
Thus 
\beq\label{1202-2}
h(\tau) \overset{\eqref{1202-1}}{<}
d_\tau(e) -\ep< d_\tau(\bar e) 
\rforal \tau\in S.
\eneq
Since $\Her(r)$ is simple and non-elementary, by inductively using Glimm's Lemma (see, for example, \cite[Proposition 4.10]{KR00}) we can obtain 
a sequence of mutually orthogonal positive elements $r_1,r_2,\cdots\in \Her(r)$ such that $\sup\{d_\tau(r_n):\tau\in S\}<\ep/4^n.$
Let $\dt_n:= \inf\{d_\tau(r_n):\tau\in S\}< \ep/4^n.$ 

{\bf Claim:} \emph{Let $c_0:=\bar e.$  There are sequences $\{a_n\}\subset (A\otimes\K)_+$
and $\{b_n\}, \{c_n\}\subset \Her(e)_+^1$ such that  for all $n\in\N,$} 

\emph{{\rm (i)} $a_n\lesssim c_{n-1}$ and $d_\tau(a_n)=h_n(\tau)$ for all $\tau\in\TQT(A);$}

\emph{{\rm (ii)} $b_n, c_n\in\Her(c_{n-1})_+^1,$ $b_nc_n=0,$ 
and $\omega(b_n)<\ep/2^n;$}

\emph{{\rm (iii)} $d_\tau(a_n)-\dt_n<d_\tau(b_n)\le d_\tau(a_n)$ for all $\tau\in S;$}

\emph{{\rm (iv)} $d_\tau(\bar e)-\sum_{i=1}^nd_\tau(a_i)
\le d_\tau(c_n)
$ for all $\tau\in S.$ }
\\
To prove the claim, we will construct $a_n, b_n, c_n$ inductively. 
By \eqref{1202-2}, 
we have  $h_1<\wh{[\bar e]}.$ Since (2)  holds, 
there is $a_1\in (A\otimes\K)_+$
such that 

(${\rm i'}$) $a_1\lesssim \bar e=c_0$ and $d_\tau(a_1)=h_1(\tau)$ for all $\tau\in\TQT(A).$ 
\\
By Proposition \ref{1204-2},  $\omega(a_1)=0.$ By Lemma  \ref{jiu16-prop1}, there are $b_1,c_1\in \Her(c_0)_+$ such that 

(${\rm ii'}$) $b_1,c_1\in \Her(c_0)_+^1,$ $b_1c_1=0,$
 and $\omega(b_1)<\ep/2;$ 

(${\rm iii'}$) $d_\tau(a_1)-\dt_1<d_\tau(b_1)\le d_\tau(a_1)$ for all $\tau\in S;$

(${\rm iv'}$) $d_\tau(\bar e)-d_\tau(a_1)
\le d_\tau(c_1)
$ for all $\tau\in S.$ 
\\
Assume that for $n\in\N$  
we have constructed $\{a_i\},\{b_i\},\{c_i\}$ ($1\le i\le n$) that
satisfy (i)-(iv). 
Let us proceed to the case $n+1:$ 
For all $\tau\in S,$ we have 
\beq\label{shi28-5}
h_{n+1}(\tau)\overset{\eqref{shi28-4}}{\le } h(\tau)-\sum_{i=1}^n h_i(\tau)
\overset{\eqref{1202-2}}{<} d_\tau({{\bar e}})- \sum_{i=1}^n h_i(\tau)
\overset{{\rm (i), (iv)}}{\le} d_\tau(c_n).
\eneq
\\
Since (2) holds, 
by \eqref{shi28-5},
there is $a_{n+1}\in (A\otimes\K)_+$
such that 

(${\rm i''}$) $a_{n+1}\lesssim c_n$ and $d_\tau(a_{n+1})=h_{n+1}(\tau)$ for all $\tau\in\TQT(A).$ 
\\
By Proposition \ref{1204-2}, 
$\omega(a_{n+1})=0.$  By Lemma \ref{jiu16-prop1}, there are $b_{n+1},c_{n+1}\in \Her(c_n)_+$ such that 

(${\rm ii''}$) $b_{n+1}, c_{n+1}\in \Her(c_n)_+^1,$ 
$b_{n+1}c_{n+1}=0,$ and $\omega(b_{n+1})<\ep/2^{n+1};$ 

(${\rm iii''}$) $d_\tau(a_{n+1})-\dt_{n+1}<d_\tau(b_{n+1})\le d_\tau(a_{n+1})$ for all $\tau\in S;$

(${\rm IV}$) $d_\tau(c_n)-d_\tau(a_{n+1})
\le d_\tau(c_{{n+1}})
$ for all $\tau\in S.$ 
\\
Then for all $\tau\in S,$ we have 
\beq
\hspace{-0.3in}
&&d_\tau(\bar e)-\sum_{i=1}^{n+1}d_\tau(a_i)
\overset{{\rm (iv)}}{\le}
d_\tau(c_n)-d_\tau(a_{n+1})
\overset{\rm (IV)}{\le} 
d_\tau(c_{{n+1}}).
\eneq
Hence we have 

(${\rm iv''}$) $d_\tau(\bar e)-\sum_{i=1}^{n+1}d_\tau(a_i) 
\le d_\tau(c_{n+1})
$ for all $\tau\in S.$ 
\\
Then (${\rm i''}$), (${\rm ii''}$), (${\rm iii''}$), (${\rm iv''}$) show that the 
$a_{n+1}, b_{n+1}, c_{n+1}$ satisfy the conditions of the claim. 
By induction, the claim holds.

Note that by (ii),  $\Her(c_0)\supset\Her(c_1)\supset\Her(c_2)\supset\cdots.$
Let $n,m\in\N$ with $n<m.$
By (ii), $b_m\in \Her(c_{m-1})\subset\Her(c_n)\subset\Her(b_n)^\perp.$ Hence 
$b_n\perp b_m$ $(n\neq m).$ 
Since $\bar e\perp r$ (see \eqref{shier10-1}) and $\{b_n\}\subset \Her(\bar e)$ and $\{r_n\}\subset \Her(r),$ we have 
\beq\label{shier10-3}
\mbox{$b_1,b_2,\cdots, r_1,r_2,\cdots$ are mutually orthogonal positive elements in $\Her(e)$.} 
\eneq
For each $n\in\N$ and for all $\tau\in S,$ we have 
\beq
h_n(\tau) \overset{\rm (i)}{=} d_\tau(a_n)\overset{\rm (iii)}{<} d_\tau(b_n)+\dt_n
\le d_\tau(b_n)+d_\tau(r_n)=d_\tau(b_n+r_n).
\eneq
Since $\TQT(A)=\R_+\cdot S,$ we have 
$
h_n(\tau) < d_\tau(b_n+r_n) \rforal \tau\in\TQT(A).
$
Since (2) holds,  
there is $x_n\in (A\otimes\K)_+$ such that $x_n\lesssim b_n+r_n$ and $d_\tau(x_n)=h_n(\tau)$ for all $\tau.$ By By Proposition \ref{1204-2} (iii), we have $\omega(x_n)=0$ for all $n.$

Since $A\otimes\K$ is stable, there are mutually orthogonal positive elements  $y_1,y_2,\cdots\in (A\otimes\K)_+^1$ such that 
$x_n\sim y_n$ and $\|y_n\|<1/2^n$ for all $n.$ 
Then $\omega(y_n)=\omega(x_n)=0$ for all $n.$
Put $y:=\sum_{n=1}^\infty y_n.$ 
Since $y_n\sim x_n\lesssim b_n+r_n$ for all $n,$ we have $\sum_{n=1}^m y_n\lesssim \sum_{n=1}^m (b_n+r_n)\lesssim a$ for all $m.$ 
Then  $y=\lim_m \sum_{n=1}^m y_n\lesssim a.$ For all $\tau\in \TQT(A),$ 
$d_\tau(y)=\sum_{n=1}^\infty d_\tau(y_n)=\sum_{n=1}^\infty d_\tau(x_n)=\sum_{n=1}^\infty h_n(\tau)=h(\tau).$ 
Thus $\Gamma$ is 
hereditary surjective. 
Moreover, by (i) and (iii),  for all $\tau\in S,$
we have $|h_n(\tau)-d_\tau(b_n)|=|d_\tau(a_n)-d_\tau(b_n)|\le \dt_n< \ep/2^n.$ Hence 
\beq\label{shier10-2}
\left|h(\tau)-\sum_{n=1}^\infty d_\tau(b_n)\right| 
\le  
\sum_{n=1}^\infty \left|h_n(\tau)-d_\tau(b_n)\right|< 
\sum_{n=1}^\infty \ep/2^n=\ep. 
\eneq
By \eqref{shier10-3}, (ii), and \eqref{shier10-2}, we see that the last statement of (3) also holds.
\end{proof} 


\begin{prop}\label{jiu21-6}
Let $A$ be a \CA. 
Let $a\in \Ped(A\otimes\K)_+^1,$  $b\in \Her(a)_+^1,$ and $\tau\in \TQT(A).$
Let $c=a^{1/2}b a^{1/2}.$
Then $\tau(a-c)\le d_\tau(a)-\tau(b).$ 
\end{prop}
\begin{proof} 
By \cite[II.2.3.]{BH}, $\tau$ is continuous on $\Her(a).$
Let $a_n:=f_{1/n}(a).$
Since $b\in \Her(a)_+^1,$ then 
$\lim_n[a_n^2, a_nba_n]=0.$ 
By  \cite[II.2.6]{BH}, 
we have 
\beq
\tau(a-c)&=&\lim_n\tau (a^{1/2}(a_n(1-b)a_n)a^{1/2})
 \le 
\lim_n \tau(a_n (1-b)a_n) 
\\&=&
\lim_n \tau(a_n^2-a_n ba_n) 
=\lim_n (\tau(a_n^2)-\tau(a_n ba_n))
=
d_\tau(a)-\tau(b). 
\eneq
\end{proof}

\begin{prop}\label{jiu23-6}
{\rm (A Dini-type result)}
Let $X$ be a compact space and let 
$f_1\ge f_2\ge f_3\ge\cdots$ be a decreasing sequence of 
real value upper semicontinuous functions on $X.$ Assume there is $r\in \R$
such that for all $x\in X,$ $\inf\{f_n(x): n\in\N\}<r.$ Then there is 
$m\in \N$ such that $f_m(x)<r$ for all $x\in X.$
\end{prop}
\begin{proof}
Define open sets $V_n:=f_n^{-1}((-\infty,r))$ for all $n\in\N.$ 
$\{f_n\}$  is decreasing implies 
$\{V_n\}$ is  increasing. 
$\inf\{f_n(x): n\in\N\}<r$ for all $x\in X$ implies 
$X=\cup_{n=1}^\infty V_n.$ $X$ is compact implies  there is $m\in\N$
such that $X=\cup_{n=1}^m V_n=V_m.$ 
Then the proposition follows. 
\end{proof}

\begin{prop}\label{shier05-1}
Let $A$ be a {{$\sigma$-unital algebraically simple}} \CA\ with 
$\QT(A)\neq\emptyset.$ 
Assume that  
$\Gamma$ is hereditary surjective. 
Let $a\in \Ped(A\otimes\K)_+$
and let $\ep>0.$  
Then there is $c\le a$ 
such that $\tau(a-c)<\ep$ 
for all $\tau\in \overline{\QT(A)}^w$
and $\omega(c)<\ep.$
\end{prop}
\begin{proof}
Without loss of generality, we may assume that $\ep<\inf\{d_\tau(a):\tau\in \ol{\QT(A)}^w\}.$ 
Note that $\wh{[a]}-\ep/4 \in \LAff_+(\TQT(A)).$
By Theorem \ref{shier09-1}, 
there is a sequence of mutually orthogonal positive elements $\{b_n\}\subset \Her(a)_+$ such that 
\beq\label{shier10-4}
\mbox{$ \omega(b_n)<\ep/2^n$ for all $n,$ and $|(d_\tau(a)-\ep/4)-\sum_{n=1}^\infty d_\tau(b_n)|<\ep/4.$}
\eneq
Define $c_m:=a^{1/2}(\sum_{n=1}^m b_n^{1/m})a^{1/2}\le a \in \Her(a)_+^1,$ then $\{c_m\}$ is an increasing sequence and $c_m\sim \sum_{n=1}^m b_n^{1/m}$ for all $m.$ 
By Proposition \ref{1204-2}, we have  
\beq
\omega(c_m)=\omega\left(\sum_{n=1}^m b_n^{1/m}\right) \le \sum_{n=1}^m\omega(b_n)\overset{\eqref{shier10-4}}{\le} \ep.
\eneq
By Proposition \ref{jiu21-6}, for all $\tau\in\ol{\QT(A)}^w,$ we have 
\beq\label{shier10-5}
\hspace{-0.2in}
\lim_m \tau(a-c_m)
&\le & 
\lim_m \left(d_\tau(a)-\tau\left(\sum_{n=1}^m b_n^{1/m}\right)\right)
\\&=& d_\tau(a)-\lim_m \sum_{n=1}^m \tau\left(b_n^{1/m}\right)
=
d_\tau(a)- \sum_{n=1}^\infty d_\tau(b_n)
\overset{\eqref{shier10-4}}{\le} 
\ep. 
\eneq 
Note that the maps $\phi_m:  \ol{\QT(A)}^w\to [0,\infty),$ $\tau\mapsto \tau(a-c_m)$ are continuous on the compact set $ \ol{\QT(A)}^w,$
and $\{\phi_m\}$ is a monotone decreasing sequence that go below $\ep$ pointwisely. Then 
by Proposition \ref{jiu23-6}, 
there is $m_0$ such that 
\beq
\tau(a-c_{m_0})<\ep
\rforal \tau\in \ol{\QT(A)}^w.
\eneq
Then $c:=c_{m_0}$  is the desired element. 
\end{proof}

\begin{thm}\label{jiu18-6}
Let $A$ be a $\sigma$-unital  simple \CA\ with 
$\QT(A)\neq\emptyset.$ 
Assume that  
$\Gamma$ is  hereditary surjective. 
Then $A$ has tracial approximate oscillation zero. 
\end{thm}
\begin{proof}
Let $e\in \Ped(A)_+^1\setminus\{0\}.$ 
Then $\Her(e)$ is algebraically simple (see Proposition \ref{shi15-1}). 
By Brown's stably isomorphism theorem (\cite{Br77}), 
$A\otimes \K\cong \Her(e) \otimes \K.$ By replacing $A$ by $\Her(e),$ we may assume that $A$ is algebraically simple. 

Let $a\in \Ped(A\otimes\K)_+^1.$ 
By Proposition \ref{shier05-1},
there is $b\le a$ 
such that $\tau(a-b)<\ep^2$ 
for all $\tau\in \ol{\QT(A)}^w,$
and $\omega(b)<\ep^2<\ep.$
Then 
$
\|a-b\|_2 
=
\sup_{\tau\in \ol{\QT(A)}^w}\tau((a-b)^2)^{1/2}
\le 
\sup_{\tau\in \ol{\QT(A)}^w}\tau(a-b)^{1/2} 
\le \ep.  
$ Thus $a$ has tracial approximate oscillation zero. 
\end{proof}

\begin{thm}
\label{jiu18-a1}
{\rm (Theorem A)}
Let $A$ be a separable
simple \CA\ of stable rank one, 
then $A$ has tracial approximate oscillation zero. 
Moreover,  $l^\infty(A)/
{{J_A}}$ has real rank zero. 
\end{thm}
\begin{proof}
For the first statement, apply Theorem \ref{shi21-3} 
and Theorem \ref{jiu18-6}. 
For the last statement, apply 
Theorem \ref{1204-6} and done.
\end{proof}

\begin{rem}
The converse of above theorem is true while assuming strict comparison, see \cite[Theorem 1.1]{FLosc}. The converse of above theorem is not true without strict comparison: Villadsen's algebra of second type has unique trace (\cite[Section 6]{V99}), hence has tracial approximate oscillation zero (\cite[Theorem 5.9]{FLosc}), but  not have stable rank one. 
\end{rem}

\section{Applications}

The readers are referred to \cite[Definition 2.1]{EHT} for the definition of diagonal AH-algebras, and are referred to 
\cite{N22}
for the definition of  (URP) and (COS). 

\begin{thm}\label{shier26-1}
Let $A$ be a simple separable unital \CA. Then $A$ has tracial approximate oscillation zero, and $l^\infty(A)/J_A$ has real rank zero, 
if one of the following holds: 

(1) $A$ is a diagonal AH-algebra. 

(2) $A= C(X)\rtimes \Z^d,$ where  $(X,\Z^d)$ is  a minimal free topological dynamical system. 

(3) $A=C(X)\rtimes \Gamma,$ where   $\Gamma$ is  a countable discrete amenable group,
and $(X,\Gamma)$ is a minimal free topological 
dynamical system with (URP) and (COS).

\end{thm}
\begin{proof}
By Theorem \ref{jiu18-a1}, it is suffices to show $A$ has stable rank one: 

(1) By \cite[Theorem 4.1]{EHT}, simple diagonal AH-algebras are have stable rank one.

(2) By  \cite[Corollary  7.9]{LN}, $C(X)\rtimes \Z^d$ has stable rank one. 

(3) By \cite[Theorem 7.8]{LN}, 
$C(X)\rtimes \Gamma$ has stable rank one. 
\end{proof}

\textsc{Xuanlong Fu}

Department of Mathematics, 
Tongji University, 
1239 Siping Road, 
Yangpu District, 
Shanghai, China, 200092. 

E-mail: xuanlongfu@tongji.edu.cn. 

\appendix
\section{Appendix}

This appendix is devoted to certain folklore results. 
We claim no originality of the results in this appendix. 

The following proposition is contained in the proof of \cite[Theorem 2.1 (iii), Corollary 2.2]{Tik14}. The original statement assumed that $A$ is $\sigma$-unital, which is not needed in here. We include a proof below for completeness. 
\begin{prop}\label{shi15-1} 
{\rm (cf.~\cite[Theorem 2.1 (iii)]{Tik14})}
Let $A$ be a simple  \CA. Then $\Ped(A)$ is algebraically simple. 
Let  $a\in\Ped(A)_+,$ then $\Her_A(a)$ is algebraically simple. 

\end{prop}

\begin{proof}
Let $x\in \Ped(A)\nzero,$ and let 
$I(x):=\{\sum_{i=1}^ny_ix^*xz_i: n\in\N, y_i,z_i\in\Ped(A), i=1,\cdots,n\}.$ 
$I(x)$ is contained in the  algebraic  ideal generated by $x$ in $\Ped(A).$ $I(x)$ is also an algebraic ideal of $A.$ $A$ is simple implies $I(x)$ is dense in $A.$ 
Then $\Ped(A)\subset I(x).$ 
It follows that $\Ped(A)=I(x)$ and  $\Ped(A)$ is algebraically simple.  

Let $a\in \Ped(A)_+\nzero,$ $b\in\Her_A(a)_+,$ and $x\in\Her_A(a)\nzero.$ Then $a^{1/3}\in \Ped(A)$ and $axa\neq 0.$ 
Since $\Ped(A)$ is algebraically simple, 
there are $y_1,\cdots y_n,z_1,\cdots,z_n\in \Ped(A)$ such that 
$b^{1/3}=\sum_{i=1}^n y_iaxa z_i.$ Then $b=\sum_{i=1}^n (b^{1/3}y_ia)x(a z_ib^{1/3}).$ Note that $b^{1/3}y_ia,a z_ib^{1/3}\in \Her_A(a)$ for all $i.$ Hence $b$ lies in the algebraic ideal generated by $x$ in $\Her_A(a).$ Thus  $\Her_A(a)$ is algebraically simple. 
\end{proof}

See Definition \ref{jiu21-d} for the topology on $\QT_2(A).$
The equivalence of (i) and (iii) in the following proposition was shown in the proof of \cite[Proposition 5.5]{ERS}. 
We mainly show here both (i) and (iii) are equivalent to (ii).  
\begin{prop}\label{jiu29-2}
{\rm (cf. \cite[Proposition 5.5]{ERS})}
Let $A$ be a simple \CA. 
Let $\{\tau_i\}\subset \QT_2(A)$ be a net and let  
$\tau\in  \QT_2(A)$. The following are equivalent:  

{\rm (i)}  $\{\tau_i\}$ converges to $\tau;$

{\rm (ii)}  $\lim_i \tau_i(a)= \tau(a)$ for all $a\in \Ped(A\otimes\K)_+;$ 

{\rm (iii)} $\lim_i \tau_i((a-\ep)_+)= \tau((a-\ep)_+)$ for all $a\in \Ped(A\otimes\K)_+$ and all $\ep>0.$ 

\end{prop}
\begin{proof}
${\rm (i)\Rightarrow (ii)}:$ 
If $\tau=\tau_\infty$ (see Definition \ref{jiu21-d}), then for all 
$a\in \Ped(A\otimes\K)_+^1,$  $\tau(a)\le \liminf_i\tau_i(a)$ implies $\lim_i\tau_i(a)=\infty=\tau(a).$ In the following, we may assume $\tau\neq\tau_\infty.$ Let $a\in \Ped(A\otimes\K)_+^1.$ 
Let $\ep\in(0,\|a\|/4).$ 
$\limsup_i\tau_i((a-\ep)_+)\le \tau(a)$ implies that 
there is $i_0$ such that for any $i\ge i_0,$ $\tau_i((a-\ep)_+)\le \tau(a)+1.$ 
Let $b:=(a-2\ep)_+\neq 0.$ 
By Proposition \ref{shi15-1}, there are 
$y_1,\cdots,y_n,z_1,\cdots,z_n \in \Ped(A)$ such that 
$a=\sum_{i=1}^ny_i  (a-2\ep)_+z_i.$ 
Let $M:=n(\tau(a)+1)/\ep.$ 
Then 
${{d_{\tau_i}(a)}}
\le n d_{\tau_i}((a-2\ep)_+) 
\le n {\tau_i}\left(\frac{ (a-\ep)_+}{\ep}\right)\le M. 
$ 
It follows that  $\|\tau_i|_{C^*(a)}\|\le  M$ for all $i\ge i_0.$ 
Hence, for any $\theta>0,$ there is $\eta>0$ such that 
$\tau_i(a)\le \tau_i((a-\eta)_+)+\theta$ for all $i\ge i_0$ 
(for example, let $\eta=\theta/M$). 
Then
\beq
\limsup_i\tau_i(a)
&\le &
\limsup_i(\tau_i((a-\eta)_+)+\theta) 
\le 
\tau(a)+\theta
\le 
\liminf_i\tau_i(a)+\theta. 
\eneq
Since $\theta$ is arbitrary, we have $\lim_i\tau_i(a)=\tau(a).$ 

${\rm (ii)\Rightarrow (iii)}$ is trivial. 
We will show ${\rm (iii)\Rightarrow (i)}.$
Let $a\in (A\otimes\K)_+$ and let $\ep>0.$ 
Then $(a-\ep)_+\in\Ped(A\otimes\K)_+,$ and   
$\limsup_i \tau_i((a-\ep)_+)
 = 
\lim_i \tau_i((a-\ep)_+)
=\tau((a-\ep)_+)\le   
\tau(a). 
$
Since $\tau$ is lower semicontinuous, 
for any $\theta>0,$
there is $\dt>0$ such that $\tau(a)-\theta \le \tau((a-\dt)_+).$ Then 
$
\tau(a)-\theta 
\le \tau((a-\dt)_+) 
= \lim_i\tau_i((a-\dt)_+) 
\le \liminf_i\tau_i(a). 
$
Since $\theta$ is arbitrary, 
we 
have $\tau(a)\le \liminf_i\tau_i(a).$ Thus 
$\{\tau_i\}$ converges to $\tau.$
\end{proof}


\begin{thebibliography}{99}
\addcontentsline{toc}{section}{Reference}

\bibitem
{APRT}
R.~Antoine,  F.~Perera,  L.~Robert, and and H.~Thiel, 
{\itshape $C^*$-algebras of stable rank one and their Cuntz semigroups.} 
Duke Math. J. {\bf 171} (2022), no. 1, 33--99.

\bibitem
{BH}B.~Blackadar and D. Handelman,
{\itshape Dimension functions and traces on $C^*$-algebra.}
J. Funct. Anal. {\bf 45} (1982), 297--340.

\bibitem{Br77}
L.~G.~Brown,  
{\itshape Stable isomorphism of hereditary subalgebras of  
$C^*$-algebras.}
Pacific J. Math. {\bf 71} (1977), no. 2, 335--348.

\bibitem{BP91}
L.~G.~Brown and G.~K.~Pedersen,
{\itshape $C^*$-algebras of real rank zero.}
J. Funct. Anal. {\bf 99} (1991), no. 1, 131--149.

\bibitem{BPT}
N.~P.~Brown, F.~Perera, and A.~Toms,
{\itshape The Cuntz semigroup, the Elliott conjecture, and dimension
functions on $C^*$-algebras.} 
J. Reine Angew. Math. {\bf 621} (2008), 191--211. 


\bibitem{CETWW21}
J.~Castillejos, S.~Evington, A.~Tikuisis, S.~White, and W.~Winter,
{\itshape Nuclear dimension of simple 
$C^*$-algebras.} 
Invent. Math. {\bf 224} (2021), no. 1, 245--290.

\bibitem{Cuntz77}
J.~Cuntz, 
{\itshape The structure of multiplication and addition in 
simple $C^*$-algebras.} Math. Scand. {\bf 40} (1977), no. 2, 215--233.

 
\bibitem{DT}
M.~Dadarlat and A.~Toms,
{\itshape Ranks of operators in simple $C^*$-algebras.} 
J. Funct. Anal. {\bf 259} (2010), 1209--1229.
 
\bibitem
{eglnp}
G.~Elliott, G.~Gong, H.~Lin, and Z.~Niu,
{\itshape Simple stably projectionless $C^*$-algebras of generalized tracial rank one.}  
J. Noncommutative Geometry,
\textbf{14} (2020), 251--347. 


\bibitem{eglnkk0} 
G.~Elliott, G.~Gong, H.~Lin, and Z.~Niu, 
{\itshape The classification of simple separable KK-contractible \CA s with finite nuclear dimension.}  
J. Geom. Phys. {\bf 158} (2020), 103861, 51 pp.

\bibitem{EHT}
G.~Elliott, T.~Ho, and  A.~Toms,  
{\itshape A class of simple  $C^*$-algebras with stable rank one.} 
J.~Funct.~Anal. {\bf 256} (2009), no. 2, 307--322. 


\bibitem{EN} G.~Elliott and Z.~Niu, 
{\itshape On the small boundary property and $\mathcal{Z}$-absorption.}   preprint, arXiv: 2504.03611v2.  


\bibitem
{ERS}
G.~Elliott, L.~Robert, and L.~Santiago, 
{\itshape The cone of lower  semicontinuous traces on a  $C^*$-algebra.} Amer. J. Math \textbf{133} (2011),   969--1005.

\bibitem
{GJS}
G.~Gong, X.~Jiang, and  H.~Su, 
{\itshape Obstructions to  ${\cal Z}$-stability for unital simple $C^*$-algebras.} 
Canad. Math. Bull. {\bf 43} (2000), no. 4, 418--426.

\bibitem{FLL21}
X.~Fu, K.~Li, and H.~Lin, 
{\itshape Tracial approximate divisibility and stable rank one.}  
J. London Math. Soc. {\bf 106} (2022), 3008--3042.

\bibitem{FL2020}
X.~Fu and H.~Lin,  
{\itshape Tracial approximation in simple $C^*$-algebras.}
 Canadian Journal of Mathematics, Volume {\bf 74} , Issue 4 , August 2022 , pp. 942--1004. 


\bibitem{FL2022}
X.~Fu and H.~Lin,   
{\itshape Nonamenable simple $C^*$-algebras with tracial approximation.} (English summary) Forum Math. Sigma {\bf 10} (2022), Paper No. e14, 50 pp. 

\bibitem
{FLosc}
X.~Fu and H.~Lin,   
{\itshape Tracial oscillation zero and stable rank one.}
Canad. J. Math. {\bf 77} (2025), no. 2, 563--630. 

\bibitem{Haa} 
U.~Haagerup, 
{\itshape Quasitraces on exact C-algebras are Traces.} 
C. R. Math. Rep. Acad. Sci. Canada Vol. {\bf 36} (2-3) 2014, pp. 67--92. 

\bibitem{JS99}X.~Jiang and  H.~Su,
{\itshape On a simple unital projectionless \CA.}
Amer. J. Math. {\bf 121} (1999), no.~{2}, 359--413.

\bibitem{KR00} E.~Kirchberg and M.~R{\o}rdam,
{\itshape Non-simple purely infinite $C^*$-algebras.}
Amer. J.  Math., {\bf 122} (2000) no.~{3}, 637--666.

\bibitem{KRadv}
E.~Kirchberg and M.~R{\o}rdam, 
{\itshape Infinite non-simple  $C^*$-algebras: absorbing the Cuntz algebras  ${\cal O}_\infty$.} 
Adv. Math. {\bf 167} (2002), no. 2, 195--264. 

\bibitem{KR14} E.~Kirchberg, and M.~R{\o}rdam,
{\itshape Central sequence $C^*$-algebras and tensorial absorption of the Jiang-Su algebra.} J. Reine Angew. Math. {\bf 695} (2014), 175--214. 

\bibitem{LN} 
C.~Li and Z.~Niu, 
{\itshape Stable rank of $C(X)\rtimes \Gamma$.}
preprint, arXiv:2008.03361v2. 2020. 


\bibitem{L22}
H.~Lin,  
{\itshape Tracial approximation and ${\cal Z}$-stability.}
preprint, arXiv: 2205.04013v3. 


\bibitem
{LinJFA}
H.~Lin,  
{\itshape Strict comparison and stable rank one.} 
J. Funct. Anal. {\bf 289} (2025), no. 9, Paper No. 111065, 25 pp. 

\bibitem
{LinAdv}
H.~Lin,  
{\itshape Tracial oscillation zero and ${\cal Z}$-stability.} 
Adv. Math. {\bf 439} (2024), Paper No. 109462, 51 pp. 

\bibitem{N22}
Z.~Niu,
{\itshape Comparison radius and mean topological dimension: Rokhlin property, comparison of open sets, and subhomogeneous $C^*$-algebras.} JAMA, {\bf 146}, 595--672 (2022).


\bibitem{MS12}
H.~Matui and Y.~Sato,
{\itshape Strict comparison and $\mathcal{Z}$-absorption of nuclear $C^*$-algebras.}
 Acta Math. {\bf 209} (2012), no.~1, 179--196.

\bibitem{MS14} H.~Matui and Y.~Sato,
{\itshape Decomposition rank of UHF-absorbing $C^*$-algebras.}
Duke Math. J., {\bf163 (14)} (2014), 2687--2708. 

\bibitem{Pedbk}  G. K. Pedersen, 
{\itshape  $C^*$-algebras and their automorphism groups.}
 London Mathematical Society Monographs, 14. Academic Press, Inc. London/New York/San Francisco, 1979.

\bibitem{Rff} 
M.~Rieffel, 
{\itshape Dimension and stable rank in the K-theory of $C^*$-algebras.} 
Proc. London Math. Soc. {\bf 46} (1983),  301--333.

\bibitem{Rob}
L.~Robert,
{\itshape Remarks on  ${\cal Z}$-stable projectionless  $C^*$-algebras.} 
Glasg. Math. J. {\bf 58} (2016), no. 2, 273--277.


\bibitem
{RorUHF2}  
M.~R{\o}rdam, 
{\itshape On the structure of
simple $C^*$-algebras tensored with a UHF-algebra, II.} 
J.~Funct.~Anal. {\bf 107} (1992), 255--269.

\bibitem{Ror04JS} M.~R{\o}rdam,
{\itshape The stable and the real rank of ${\cal Z}$-absorbing $C^*$-algebras.}
Internat. J. Math. {\bf 15} (2004), no.~{10}, 1065--1084.

\bibitem{T20}
H.~Thiel, 
{\itshape Ranks of operators in simple $C^*$-algebras with stable rank one.} 
Comm. Math. Phys. {\bf 377} (2020), no. 1, 37--76. 

\bibitem{Tik14}
A.~Tikuisis, 
{\itshape Nuclear dimension, ${\cal Z}$-stability, and algebraic simplicity
for stably projectionless $C^*$-algebras.}
Math. Ann. (2014) {\bf 358}, 729--778. 

\bibitem{V98}
J.~Villadsen, 
{\itshape Simple  $C^*$-algebras with perforation.}
J. Funct. Anal. {\bf 154} (1998), no. 1, 110--116.

\bibitem{V99}
J.~Villadsen, 
{\itshape On the stable rank of simple $C^*$-algebras.} 
J. Amer. Math. Soc. {\bf 12} (1999), no. 4, 1091--1102. 

\bibitem{W12pure} W.~Winter,
{\itshape Nuclear dimension and $\mathcal{Z}$-stability of pure $C^*$-algebras.}
Invent. Math. {\bf 187} (2012), no.~{2}, 259--342.

\bibitem{Z90}
S.~Zhang, 
{\itshape A property of purely infinite simple $C^*$-algebras.}
Proc. Am. Math. Soc. 
Vol. {\bf 109}, no. 3 (Jul., 1990), pp.~717--720.

\end{thebibliography}
\end{document}